\documentclass[12pt]{amsart}

\usepackage{amsmath,amsthm,amssymb}
\usepackage{mathrsfs}
\usepackage{cite}

\makeatletter
\newcommand{\subjclassname@later}{\textup{2010} Mathematics Subject Classification}
\makeatother

\newtheorem{theorem}{Theorem}[section]
\newtheorem{lemma}{Lemma}[section]

\newtheorem{remark}{Remark}[section]

\numberwithin{equation}{section} \numberwithin{theorem}{section}

\DeclareMathOperator{\RE}{Re}

\allowdisplaybreaks

\begin{document}

\title{The third logarithmic coefficient for the subclasses of
close-to-convex functions}
\author[N. M. Alarifi]{Najla M. Alarifi}
\address{Department of Mathematics, Imam Abdulrahman Bin Faisal University, Dammam 31113, Kingdom of Saudi Arabia}
\email{najarifi@gmail.com}

\begin{abstract}
Let $\mathcal{A}$ denote the set of all analytic functions $f$ in the unit disk $\mathbb{D}:=\{z \in \mathbb{C}: |z| < 1\}$ normalized by $f (0) = 0$ and $f'(0) = 1.$
The logarithmic coefficients $\gamma_n$ of $f \in \mathcal{A}$
are defined by $  \log f(z)/z =2 \sum_{n=1}^{\infty}\gamma_{n}z^{n}.$ In the present paper,  the upper bound of the third logarithmic
 coefficient in general case of $f''(0)$  was computed
 when $f$ belongs to some familiar subclasses of close-to-convex functions.
\end{abstract}



\maketitle

\section{Introduction and Preliminaries}
Let $\mathbb{D}:=\{z \in \mathbb{C}: |z| < 1\}$ denote the open unit disk in the complex plane $\mathbb{C}.$
Let $\mathcal{A}$ denote the set of all analytic functions $f$ in the open unit disk of the form
\begin{equation}\label{eq1.1}
f(z) =\sum_{n=1}^{\infty} a_{n}z^{n}, \quad with \quad a_1=1,
\end{equation}
and $\mathcal{S}$ be its subclass consisting functions that are univalent in $\mathbb{D}.$
Given a function $f \in \mathcal{S},$ the coefficients $\gamma_{n}$ defined by
\begin{equation}\label{eq1.2}
 \log \frac{f(z)}{z}=2 \sum_{n=1}^{\infty}\gamma_{n}z^{n},\quad  z\in \mathbb{D}\setminus \{0\} ,\quad \log 1 :=0.
\end{equation}
are called the logarithmic coefficients of $f$.
As it is commonly known, the logarithmic
coefficients takes a leading role in Milin conjecture (\cite{Milin}, \cite[p.155]{Duren}), that is $f \in \mathcal{S},$
 \begin{equation*}
\sum_{m=1}^{n}  \sum_{k=1}^{m}\left(k |\gamma_k|^2- \frac{1}{k} \right)\leq 0.
\end{equation*}

 Milin conjecture was confirmed to be the famous Bieberbach conjecture (e.g., \cite[p.37]{Duren}) by De Branges  \cite{ Branges}.
Sharp estimates for the class
 $\mathcal{S}$ are known only for the first two coefficients,
 \begin{equation*}
  | \gamma_1 | \leq 1 ,\quad | \gamma_2 | \leq  \frac{1}{2}+\frac{1}{e}=0.635\cdots.
\end{equation*}
However, Obradovi\'c and Tuneski \cite{Tuneski} obtained an upper bound of $| \gamma_3 |$
 for the class $\mathcal{S}$.

The problem of estimating the modulus of the first
three logarithmic coefficients is significantly studied for the subclasses of
$\mathcal{S}$
and in some
cases sharp bounds are obtained. For instance, sharp estimates for the class of starlike functions $ \mathcal{S^*} $ are given by
 the inequality $|\gamma_n|\leq 1/n$ holds for $n \in  \mathbb{N}$
 \cite[p.42]{Thomas4}. 
Furthermore, for $f \in \mathcal{S} \mathcal{S^*}, $
 the class of strongly starlike function of order $ \beta,$
$(0 \leq\beta\leq 1),$ it holds that $|\gamma_n|\leq \beta/n$  $(n \in  \mathbb{N})$ \cite{ Thomas2}.
 The bounds of $ \gamma_n $  for functions in subclasses of
 $\mathcal{S}$ has been widely studied in recent years, sharp estimates for the class of strongly
starlike functions of certain order and
gamma-starlike functions and different subclasses of close-to-convex functions are given in \cite{ Thomas2}, \cite[p.116]{Thomas4} and \cite{Thomas3}
respectively, while non-sharp estimates for the class of Bazilevic, close-to-convex
 are given in \cite{Deng,Thomas1,Ali}
respectively.

Let $\mathcal{F}_1, \mathcal{F}_2, \mathcal{F}_3$ denote the subclasses of $\mathcal{S}$  satisfying respectively the next condition
\begin{align*}
& \RE \{(1-z)f'(z)\} >0 ,\quad  z\in \mathbb{D}, \\
& \RE \{(1-z^2)f'(z) \} >0 ,\quad   z\in \mathbb{D}, \\
 &\RE \{(1-z+z^2)f'(z) \} >0 ,\quad   z\in \mathbb{D} .\\
\end{align*}
Account that each class defined above is the subclass of the well known
 class of close-to-convex functions, consequently families $\mathcal{F}_i,$ $i = 1,2 , 3,$ contain only
 univalent functions \cite[Vol.II,p. 2]{Duren}.
The sharp bounds of $ \gamma_1, $
 $ \gamma_2 $ and partial results for $ \gamma_3 $
of the subclasses $\mathcal{F}_1, \mathcal{F}_2, \mathcal{F}_3$ of $\mathcal{S}$ were determined by Kumar and Ali  \cite{ Kumar}.
 Moreover,  Cho \emph{et al.}  \cite{Cho2020} computed the sharp upper bounds for the third logarithmic coefficient $\gamma_3$ of $f$ when $a_2$ is real number.

Differentiating \eqref{eq1.2} and comparing the coefficients with \eqref{eq1.1}, we get $\gamma_1=\frac{1}{2} a_1,$ $\gamma_2=\frac{1}{2}\left(a_3-\frac{1}{2}a^2\right)$
 and
\begin{equation}\label{eq1.3}
\gamma_3=\frac{1}{2}\left(a_4-a_2a_3+\frac{1}{3} a^3_2\right).
\end{equation}
The main aim of this paper
is to determine upper bound of the third logarithmic
 coefficient in general case of $a_2.$

The following lemma are needed to prove our main results
\begin{lemma}\label{lem1}{\rm \cite{Carlson}}
Let $ w(z) =c_1z+c_2z^2+\cdots $ be a Schwarz function. Then
\begin{equation*}
 |c_1|\leq 1,\quad |c_2|\leq  1-|c_1|^2 \quad   and \quad   |c_3|\leq  1-|c_1|^2-\frac{|c_2|^2}{1+|c_1|}.
\end{equation*}
\end{lemma}

 \section{Main Results}
\begin{theorem}\label{thm2.1}
Let $f\in \mathcal{F}_1.$ Then
\begin{equation*}
|\gamma_3|\leq\frac{15.75}{48}=0.328125.
\end{equation*}

\end{theorem}

\begin{proof}
Since  $f\in  \mathcal{F}_1,$ an analytic self-map $w$ of $\mathbb{D}$ with $w(0) = 0$ exists and
\begin{equation}\label{eq2.a}
(1-z)f'(z)=\frac{1+w(z)}{1-w(z)}=1+2w(z)+2w^2(z)+\cdots.
\end{equation}
Writing
\begin{equation}\label{eq2.2}
w(z) =c_1z+c_2z^2+\cdots,
\end{equation}
then by using \eqref{eq2.a} along with \eqref{eq2.2} lead to
\begin{align}\label{eq2.3}
&a_2= \frac{1}{2}(1+ 2c_1),\nonumber\\
&a_3= \frac{1}{3}(1+ 2c_1+2 c^2_1+2c_2),\nonumber\\
&a_4= \frac{1}{4}(1+ 2c_1+2c_2+2 c_3+2 c^2_1+4c_1c_2+2 c^3_1).
\end{align}
From \eqref{eq1.3} and \eqref{eq2.3} after some calculations, the following was obtained
\begin{equation*}
\gamma_3=\frac{1}{48}(3+ 2c_1+4c_2+12 c_3+8c_1c_2+4 c^3_1),
\end{equation*}
and from here by using Lemma  \ref{lem1}
\begin{align}\label{eq2.4}
48|\gamma_3|&\leq 3+ 2|c_1|+4|c_2|+12| c_3|+8|c_1||c_2|+4 |c_1|^3\nonumber\\
&\leq 3+ 2|c_1|+4|c_2|+12\left(1-|c_1|^2-\frac{|c_2|^2}{1+|c_1|}\right)+8|c_1||c_2|+4 |c_1|^3=:f_1(|c_1|,|c_2|),
\end{align}
where
   \begin{align*}\label{eq2.5}
&f_1(x,y)=3+ 2x+4y+12\left(1-x^2-\frac{y^2}{1+x}\right)+8x y+4 x^3,\nonumber\\
&(x,y)\in E:\quad 0\leq x \leq 1,\quad  0\leq y \leq 1-x^2.
\end{align*}
The system
\begin{align*}
&\frac{\partial f_1(x,y)}{\partial x}=2-24x +12\left(\frac{y}{1+x}\right)^2+8 y+12 x^2 =0,\\
&\frac{\partial f_1(x,y)}{\partial y}=4-\frac{24y}{1+x} +8x =0
\end{align*}
has unique solution $(x_1,y_1)=(1/4,5/16)\in E\setminus \partial E $
with
\begin{equation}\label{eq2.6}
 f_1(x_1,y_1)=15.75.
 \end{equation}
The $\max f_1(x_1,y_1) ,$ needs to be found when $ (x,y) $ belongs to the boundary of $E.$ In that sense, we have
\begin{align}\label{eq2.7}
&f_1(x,0)=15+2 x-12x^2+4x^3\leq 9+ \frac{10\sqrt{30}}{9}=15.08580\cdots \quad for \quad 0\leq x \leq 1,\nonumber\\
&f_1(0,y)= 15+4 y-12 y^2\leq \frac{46}{3}=15.33\cdots \quad for \quad 0\leq y \leq 1,\nonumber\\
%\end{align}
%and
%\begin{equation}\label{eq2.8}
&f_1(x,1-x^2)=7+22x-4x^2-16x^3\leq 15.304035\cdots.
\end{align}
Using \eqref{eq2.4}, \eqref{eq2.6}  and \eqref{eq2.7}, concludes that
\begin{equation*}
48| \gamma_3| \leq 15.75,\quad i.e ,\quad |\gamma_3| \leq 0.328125.
\end{equation*}
This completes the proof.
\end{proof}
\begin{remark}
If $f\in \mathcal{F}_1,$ where $f''(0)$ is real number, then \cite{Cho2020}
\begin{equation*}\label{eq2.1}
|\gamma_3|\leq \frac{1}{288}\left(11+15\sqrt{30}\right) =0.323466\cdots.
\end{equation*}
\end{remark}

\begin{theorem}\label{thm2.2}
Let $f\in \mathcal{F}_2.$ Then
\begin{equation*}\label{eq2.1}
|\gamma_3|\leq 0.258765\cdots.
\end{equation*}

\end{theorem}

\begin{proof}
Since $f\in  \mathcal{F}_2,$ there exists an analytic self-map $w $  of $\mathbb{D}$ with $w(0) = 0$ and
\begin{equation}\label{eq2.9}
(1-z^2)f'(z)=\frac{1+w(z)}{1-w(z)}=1+2w(z)+2w^2(z)+\cdots.
\end{equation}
The coefficients can be found by comparing the notations given in \eqref{eq2.2} and \eqref{eq2.9}
\begin{align}\label{eq2.10}
&a_2= c_1,\nonumber\\
&a_3= \frac{1}{3}(1+ 2c_2+2 c^2_1 ),\nonumber\\
&a_4= \frac{1}{2}(c_1+c_3+2 c_1 c_2+ c^3_1).
\end{align}
From \eqref{eq1.3} and \eqref{eq2.10} after some calculations, the following was obtained
\begin{equation*}
\gamma_3=\frac{1}{12}(c_1+3c_3+ 2 c_1c_2+ c^3_1),
\end{equation*}
and from here by using Lemma  \ref{lem1}
\begin{align}\label{eq2.11}
12| \gamma_3|&\leq  |c_1| +3| c_3|+2|c_1||c_2|+|c_1|^3\nonumber\\
&\leq |c_1|+ 3 \left(1-|c_1|^2-\frac{|c_2|^2}{1+|c_1|}\right)+2|c_1||c_2|+|c_1|^3=:f_2(|c_1|,|c_2|),
\end{align}
where
   \begin{align*}\label{eq2.12}
&f_2(x,y)=x+3 \left(1-x^2-\frac{y^2}{1+x}\right)+2xy+ x^3,\nonumber\\
&(x,y)\in E:\quad 0\leq x\leq 1,\quad  0\leq y \leq 1-x^2.
\end{align*}
From the system
\begin{align*}
&\frac{\partial f_2(x,y)}{\partial x}=1-6x +3\left(\frac{y}{1+x}\right)^2+2y+ 3 x^2 =0,\\
&\frac{\partial f_2(x,y)}{\partial y}=-\frac{6y}{1+x} +2x =0,
\end{align*}
only one solution $(x_2,y_2)$ lies in the interior of $E ,$ where
\begin{align*}
&x_2= \frac{4-\sqrt{7}}{6} =0.22570 \cdots,\nonumber\\
&y_2= \frac{47-14\sqrt{7}}{108}=0.092217\cdots ,
\end{align*}
and
\begin{equation}\label{eq2.13}
 f_2(x_2,y_2)= 3.10518\cdots.
 \end{equation}
On the boundary of $E,$ we have the next property
\begin{align}\label{eq2.14}
&f_2(x,0)=3(1-x^2)+x+x^3\leq 2+\frac{4}{9}\sqrt{6} =3.08866 \quad for \quad 0\leq x \leq 1,\nonumber\\
&f_2(0,y)=3 (1-y^2) \leq 3 \quad for \quad 0\leq y \leq 1,\nonumber\\
&f_2(x,1-x^2)=6x-4x^3\leq 2 \sqrt{2} =2.82842\cdots.
\end{align}
Consequently \eqref{eq2.11}, \eqref{eq2.13} and \eqref{eq2.14} yield
\begin{equation*}
12|\gamma_3| \leq 3.10518\cdots,\quad i.e ,\quad |\gamma_3| \leq 0.258765\cdots.
\end{equation*}
\end{proof}
\begin{remark}
If $f\in \mathcal{F}_2,$ where $f''(0)$ is real number, then \cite{Cho2020}
\begin{equation*}\label{eq2.1}
|\gamma_3|\leq \frac{1}{972}\left(95+23\sqrt{46}\right) =0.258223\cdots.
\end{equation*}
\end{remark}

\begin{theorem}\label{thm2.3}
Let $f\in \mathcal{F}_3.$ Then
\begin{equation*}\label{eq2.1}
|\gamma_3|\leq \frac{17.75}{48}= 0.36979\cdots.
\end{equation*}

\end{theorem}

\begin{proof}
Proceeding similarly as in the previous proofs, there exists an analytic self-map $w$ of $\mathbb{D}$ with $w(0) = 0$ and
\begin{equation}\label{eq2.15}
(1-z+z^2)f'(z)=\frac{1+w(z)}{1-w(z)}=1+2w(z)+2w^2(z)+\cdots.
\end{equation}
Substituting \eqref{eq2.2} into \eqref{eq2.15}
and after comparing coefficients leads to
\begin{align}\label{eq2.16}
&a_2=\frac{1}{2}(1+2 c_1),\nonumber\\
&a_3= \frac{2}{3}(c_1+ c_2+ c^2_1 ),\nonumber\\
&a_4= \frac{1}{4}(2c_2+2c_3+2 c^2_1+2 c^3_1+4c_1 c_2-1).
\end{align}
By using \eqref{eq1.3}  along with \eqref{eq2.16}, upon simplification
\begin{equation*}
\gamma_3=\frac{1}{48}(-5-2c_1+4c_2+12c_3+8 c_1c_2+4 c^3_1),
\end{equation*}
and from here by using Lemma  \ref{lem1}
\begin{align}\label{eq2.17}
48|\gamma_3|&\leq 5+2 |c_1| +4| c_2|+12|c_3|+8|c_1||c_2|+4|c_1|^3\nonumber\\
&\leq 5+2|c_1|+4|c_2|+12 \left(1-|c_1|^2-\frac{|c_2|^2}{1+|c_1|}\right)+8|c_1||c_2|+4|c_1|^3=:f_3(|c_1|,|c_2|),
\end{align}
where
   \begin{align*}\label{eq2.18}
&f_3(x,y)=5+2x+4y+12 \left(1-x^2-\frac{y^2}{1+x}\right)+8xy+4x^3,\nonumber\\
&(x,y)\in E:\quad 0\leq x \leq 1,\quad  0\leq y \leq 1-x^2.
\end{align*}
The system
\begin{align*}
&\frac{\partial f_3(x,y)}{\partial x}=2-24x +12\left(\frac{y}{1+x}\right)^2+8 y+12 x^2 =0,\\
&\frac{\partial f_3(x,y)}{\partial y}=4-\frac{24y}{1+x} +8x =0
\end{align*}
has unique solution $(x_3,y_3)=(1/4,5/16) $ belongs to the interior of $E$
 and
\begin{equation}\label{eq2.19}
 f_1(x_3,y_3)=17.75.
 \end{equation}
On the boundary of $E,$ the following cases are given
\begin{align}\label{eq2.20}
&f_3(x,0)=17+2 x-12x^2+4x^3\leq 11+ \frac{10\sqrt{30}}{9}=17.08580\cdots \quad for \quad 0\leq x \leq 1,\nonumber\\
&f_3(0,y)= 17+4 y-12 y^2\leq  17.33\cdots \quad for \quad 0\leq y \leq 1,\nonumber\\
%\end{align}
%and
%\begin{equation}\label{eq2.21}
&f_3(x,1-x^2)=9+22x-4x^2-20x^3\leq 16.56455\cdots.
\end{align}
Equations \eqref{eq2.17}, \eqref{eq2.19} and \eqref{eq2.20} show that
\begin{equation*}
 | \gamma_3| \leq \frac{17.75}{48}= 0.36979\cdots.
\end{equation*}
\end{proof}
\begin{remark}
Let $f\in \mathcal{F}_3,$ where $f''(0)$ is real number. Then \cite{Cho2020}
\begin{equation*}\label{eq2.1}
|\gamma_3|\leq \frac{1}{7776}\left(743+131\sqrt{262}\right) =0.368238\cdots.
\end{equation*}
\end{remark}

%%%%%%%%%%%%%%%%%%%%%%%%%%%%%%%%%%%%%%%%%%%%%%%%%%%%%%%%%%%

\end{document}